\documentclass[a4paper,12pt]{amsart}

\usepackage{lmodern}
\usepackage[english]{babel}

\usepackage{amsmath,amssymb,mathrsfs}

\usepackage{bookmark}
\usepackage{hyperref}

\newtheorem{thm}{Theorem}[section]

\newtheorem{lemma}[thm]{Lemma}
\newtheorem{prop}[thm]{Proposition}

\theoremstyle{definition}
\newtheorem{defn}[thm]{Definition}
\newtheorem{remark}[thm]{Remark}

\newcommand{\N}{\mathbb N}
\newcommand{\HO}{H_0^1(\Omega)}
\newcommand{\Dud}{\mathcal{D}^{1,2}(\mathbb{R}^N)}

\newcommand{\setR}{{\mathbb{R}^N}}
\newcommand{\eps}{\varepsilon}
\newcommand{\di}{\mathop{}\!\mathrm{d}}

\newcommand{\Neps}{\mathcal{N}_\eps}
\newcommand{\Peps}{p_\eps}

\newcommand{\Ds}{2_\mu^*}

\let\epsilon\varepsilon

\DeclareMathOperator{\dist}{dist}

\usepackage[initials]{amsrefs}

\begin{document}

\title[Multiple positive solutions for a Choquard problem]{Multiple positive solutions for a slightly subcritical Choquard problem on bounded domains}

\author{Marco Ghimenti}
\author{Dayana Pagliardini}

\address{Marco Ghimenti \newline
Dipartimento di Matematica 
Universit\`a di Pisa
Via F. Buonarroti 1/c, 
56127 Pisa, Italy}
\email{marco.ghimenti@unipi.it}

\thanks{The first author is partially supported by PRA Universit\`a di Pisa}

\address{Dayana Pagliardini \newline
Scuola Normale Superiore,
Piazza dei Cavalieri 7,
56126 Pisa, Italy}
\email{dayana.pagliardini@sns.it}

\begin{abstract}
In this paper we study a slightly subcritical Choquard problem on a bounded domain $\Omega$. We prove that the number of positive solutions depends on the topology of the domain. In particular when the exponent of the nonlinearity approaches the critical one, we show the existence of cat$(\Omega)+1$ solutions. Here 
cat$(\Omega)$ denotes the Lusternik-Schnirelmann category.
\end{abstract}

\date{\today}


\maketitle

\section{Introduction}
In the last few years a lot of mathematical efforts have been devoted to the Choquard type equation
\begin{equation}\label{choquard}
-\Delta U+V(x)u=(|x|^{-\mu}* |u|^q)|u|^{q-2}u\quad \text{in}\; \setR
\end{equation}
where $N\ge 3$, $0 < \mu < N$ and $V(x)$ is the external potential. This equation was introduced by Pekar to study the quantum theory of a polaron at rest. Then Choquard applied it to model an electron trapped in its own hole, in an approximation to the Hartree-Fock theory of one-component plasma \cite{Lieb}. In some cases equation $\eqref{choquard}$ is known also as the Schr\"odinger-Newton equation. There are a lot of studies about the existence, multiplicity and properties of solutions of the nonlinear Choquard equation $\eqref{choquard}$. In particular, if $N = 3$, $q = 2$, $\mu = 1$, and the potential is constant, the existence of ground states of equation $\eqref{choquard}$ was established in \cite{Lieb} and \cite{Lions} through variational methods, while uniqueness and nondegeneracy were obtained respectively in \cite{Lieb} and \cites{L, WW}.

For general $q$ and $\mu$, regularity, positivity, radial symmetry and decay property of the ground states were shown in \cites{CCS, MZ, MV}. In \cite{MV'} Moroz and Van Schaftingen proved the existence of positive ground states under the assumptions of Berestycki-Lions type, while the existence of sign changing solutions of the nonlinear Choquard equation was considered in \cites{CS, GMV, GV}.
Recently, in \cite{GY}, Gao and Yang established some existence result for the Brezis-Nirenberg type problem of the nonlinear Choquard equation, i.e.\
\begin{equation}\label{BN}
-\Delta u-\lambda u=\Big(\int_\Omega\dfrac{|u|^{\Ds}(x)}{|x-y|^\mu}\di x \Big)|u|^{\Ds -2}u\quad \text{in}\; \Omega,
\end{equation}
 where $\Omega \subset \setR$ is a bounded domain with Lipschitz boundary, $\lambda > 0$, $N\ge 3$ and $\Ds:=(2N-\mu)/(N-2)$ with $0<\mu<N$.

In this paper we are interested in the slightly subcritical Choquard problem 
\begin{equation}\label{P}
\begin{cases}
-\Delta u-\lambda u=\Big(\displaystyle\int_\Omega\dfrac{|u|^{p_\eps}(x)}{|x-y|^\mu}\di x \Big)|u|^{p_\eps -2}u\qquad &\text{in}\; \Omega\\
u=0\qquad &\text{in}\; \mathbb{R}^N \setminus \Omega\\
u>0\qquad &\text{in}\; \Omega,
\end{cases}
\end{equation}
where $\eps>0$, $\Omega$ is a regular bounded domain of $\mathbb{R}^N$, $N > 3$, $0<\mu<N$, $\lambda \ge 0$ and $p_\eps:=2_\mu^*-\eps$, with $2_\mu^*$ that is the critical exponent defined as $2_\mu^*:=\dfrac{2N-\mu}{N-2}$.

Our goal is to show that the number of positive solutions of nonlinear Choquard equation depends on the topology of the domain when the exponent is very close to the critical one, i.e.\ we want to prove the following:
\begin{thm}\label{Mainthm}
There exists $\bar{\eps}>0$ such that for every $\eps \in (0, \bar{\eps}]$, Problem $\eqref{P}$ has at least cat$_{\Omega}(\Omega)$ low energy solutions. Moreover, if 
$\Omega$ is not contractible, there exists another solution with higher energy.
\end{thm}

Here cat$_{\Omega}(\Omega)$ denotes the Lusternik-Schnirelmann category of $\Omega$.

To tackle Problem $\eqref{P}$ we use variational methods and hence we look for critical points of a suitable functional and we prove a multiplicity result through category methods. If cat$_{\Omega}(\Omega)=1$ we get the existence of a solution (which can be obtained in a simpler way by the Mountain Pass Theorem); on the other side, if $\Omega$ is not contractible, we obtain cat$_{\Omega}(\Omega)$ low energy solutions, and another solution with higher energy. For this reason, hereafter, we assume cat$_{\Omega}(\Omega)>1$.

This type of result was historically introduced by Bahri and Coron for local problems in \cite{BaC} and recovered by Benci and Cerami in \cite{BC} and by Benci, Cerami and Passaseo in \cite{BCP}.

An application of these methods to nonlocal problems is exhibited by Siciliano in \cite{S}, which investigated the existence of positive solutions to the Schr\"odinger-Poisson-Slater system 
\begin{equation}
\begin{cases}
-\Delta u+u+l \Phi u=|u|^{p-2}u\quad &\text{in}\; \Omega\\
-\Delta \Phi=u^2 \quad &\text{in}\;\Omega\\
u=\Phi=0\quad &\text{on}\;\partial \Omega,
\end{cases}
\end{equation}
where $\Omega \subset \mathbb{R}^3$ is a bounded domain, $l>0$ is a fixed parameter and $p<2^*$. Siciliano proved that, when the power exponent is near to the critical Sobolev exponent, the number of positive solutions is greater than the Lusternik-Schnirelmann category of $\Omega$.

Also, in the same spirit, in \cite{AGSY}, Alves, Gao, Squassina and Yang studied the semiclassical limit for the singularly perturbed Choquard equation
\[
-\eps^2\Delta u+V(x)u=\eps^{\mu-3}\Big(\int_{\mathbb{R}^3}\frac{Q(y)G(u(y))}{|x-y|^\mu}\di y \Big)Q(x)g(u)\quad \text{in}\; \mathbb{R}^3
\]
with $0<\mu<3$, $\eps>0$, $V$ and $Q$ that are two continuous real functions on $\mathbb{R}^3$ and $G$ that is the primitive of $g$ which is of critical growth. They proved, under suitable assumptions on $g$, the exitence of ground states for the critical Choquard equation with constant coefficients, the existence and multiplicity of semi-classical solutions, and they used variational methods to characterize the concentration behavior.

To prove Theorem \ref{Mainthm} we construct a map $\psi_\eps$ from $\Omega_r^-$ to a suitable subspace $\mathcal{N}\subset H^1_0(\Omega)$ 
and a function $\beta:\mathcal{N}\rightarrow \Omega_r^+$, defined respectively in (\ref{psii}) and in (\ref{eq:bari}), where 
\begin{align*}
\Omega_r^+:&=\{x\in \mathbb{R}^N :\di (x,\Omega)\le r \},\\
\Omega_r^-:&=\{x\in \Omega :\di (x,\partial \Omega)\ge r \}
\end{align*}
and $r>0$ is chosen sufficiently small such that both sets are homotopically equivalent to $\Omega$. Then a classical result in critical point theory 
(recalled in Section \ref{sec2}) gives us the existence of cat$_{\Omega}(\Omega)$ solutions of (\ref{P}). Also, we construct a compact and contractible set 
$T_\eps\subset H^1_0(\Omega)$ which gives the existence of another solution if $\Omega$ is not contractible. The variational structure of this problem and 
the functional setting are introduced in in Section \ref{sec3}, while the main Theorem is proved in Section \ref{sec5}. 
Also, in section \ref{sec4} we present two limit cases of the Problem $\eqref{P}$ whose study is useful to achieve the proof.

\section{Notation and preliminaries}\label{sec2} 

Given $u:\mathbb{R}^N \rightarrow \mathbb{R}$ a measurable function, for $p\in [1,+\infty)$ we denote with $| u |_{L^p(\Omega)}$ the standard $L^p$-norm of the function $u$ in $\Omega$. When the domain is clear we use the notation $|u|_p$.

As usual $\mathcal{D}^{1,2}(\Omega)$ is the completion of $C_0^\infty(\Omega)$ with respect to the norm 
$\|u \|^2_{\mathcal{D}^{1,2}(\Omega)}:={\int_\Omega |\nabla u |^2\di x}$, $H_0^1(\Omega)$ is the Sobolev space with squared norm
\[
\|u\|^2=|\nabla u|_2^2+|u|_2^2
\]
and $H^{-1}$ is its dual. Since $\lambda \ge 0$,
\[
\|u\|_\lambda:=|\nabla u|_2^2+\lambda |u|_2^2
\]
is an equivalent norm in $H_0^1(\Omega)$.

We use $B_\rho(y)$ for the closed ball of center $y$ and radius $\rho>0$. When $y=0$ we will simply write $B_\rho$.

Hereafter we will use the letter $c$ to denote a positive constant whose value can be different from a line to another.

We remind some preliminary results which will be useful in the sequel (see \cite{c}).

\begin{defn}\label{PalaisSmale}
Let $H$ be a Hilbert space and $I\in C^1(H)$. We say that $I$ satisfies the $(PS)$-condition on $H$ if every subsequence $\{u_n\}_{n\in \mathbb{N}}$ such that
\begin{equation}\label{PScondition}
\{I(u_n)\} \; \text{is bounded}\quad \text{and}\quad I'(u_n)\rightarrow 0\quad \text{in}\; H^{-1}(\Omega),
\end{equation}
has a converging subsequence. If a sequence satisfies $\eqref{PScondition}$ it is called Palais-Smale sequence (or $(PS)$-sequence).
\end{defn}

\begin{defn}\label{category}
Let $X$ be a topological space. The category of $A\subseteq X$ with respect to $X$, denoted by cat$_X(A)$, is the least integer $k$ such that $A\subseteq A_1 \cup \cdots \cup A_k$ with $A_i$ closed and contractible in $X$ for every $i=1,\cdots, k$.
\end{defn}

We set cat$_X(\varnothing)=0$, cat$_X(A)=+\infty$ if there are no integers with the above property and we will write cat$(X)$ for cat$_X(X)$.

\begin{remark}\label{catinclusione} [see \cite{BC}]
Let $X$ and $Y$ be topological spaces. If $f:X \rightarrow Y$, $g:Y \rightarrow~X$ are continuous operators such that $ g \circ f$ is homotopic to the identity on $X$, then cat$_X(X)\le$ cat$_X(Y)$.
\end{remark}

For the proof of the main theorem we will use the following result (see \cite{BCP}):
\begin{thm}\label{LSthm}
Let $J$ a $C^1$ real functional on a complete $C^{1,1}$ Banach manifold $M$. If $J$ is bounded from below and satisfies the $(PS)$-condition on $M$, then $J$ has at least cat$(J^d)$ critical points in $J^d$, where $J^d:=\{ u\in M: J(u)<d\}$. Moreover if $M$ is contractible and cat$(J^d)>1$, there is at least one critical point $u\notin J^d$.
\end{thm}
Finally we briefly recall the well-known Hardy-Littlewood-Sobolev inequality and some results linked to it (see \cite{LL} and \cite{GY}):
\begin{prop}\label{HardyLittS}
Let $t,r>1$ and $0<\mu<N$ with $1/t+\mu/N+1/r=2$, $f\in L^t(\setR)$ and $u\in L^r(\setR)$. There exists a constant $C(t,N,	\mu,r)$, independent of $f$ and $h$, such that
\begin{equation}\label{HLS}
\int_\setR \int_\setR \frac{f(x)u(y)}{|x-y|^\mu}\di x \di y\le C(t,N,\mu,r)|f|_t|u|_r.
\end{equation}
Equality in $\eqref{HLS}$ holds if and only if 
\begin{equation}\label{eqHLS}
f(x)=U_{R,a}(x):=C\Big(\frac{R}{1+R^2|x-a|^2}\Big)^{\frac{N-2}{2}},
\end{equation}
where $C:=(N(N-2))^{\frac{N-2}{4}}$, $a\in \setR$ and $R\in (0,+\infty)$.
\end{prop}
So, for all $u\in \mathcal{D}^{1,2}(\setR)$ we have
\[
\Big(\int_\setR \int_\setR \frac{|u(x)|^{\Ds}|u(y)|^{\Ds}}{|x-y|^\mu}\di x \di y \Big)^{\frac{N-2}{2N-\mu}}\le C(N,\mu)^{\frac{N-2}{2N-\mu}}|u|_{2^*}^2
\]
and denoting with $S_{H,L}$ the best constant defined by
\begin{equation}\label{SHL}
S_{H,L}:=\inf_{u\in \mathcal{D}^{1,2}(\setR)\setminus \{ 0\}}\frac{\displaystyle \int_\setR |\nabla u|^2\di x}{\displaystyle\Big(\int_\setR \int_\setR \frac{|u(x)|^{\Ds}|u(y)|^{\Ds}}{|x-y|^\mu}\di x \di y\Big)^{\frac{N-2}{2N-\mu}}},
\end{equation}
we note from Proposition \ref{HardyLittS} that it is achieved if and only if $u=U_{R,a}$. In addition
\[
S_{H,L}=\frac{S}{C(N,\mu)^{\frac{N-2}{2N-\mu}}}
\]
where $S$ is the best Sobolev constant.

\begin{lemma}\label{Shlomeganotachiev}\cite{GY}*{Lemma $1.3$}
Let $N\ge 3$. For every $\Omega\subset \setR$ open, we have
\begin{equation}
S_{H,L}^\Omega:=\inf_{u\in \mathcal{D}_0^{1,2}(\Omega)\setminus \{ 0\}}\frac{\displaystyle\int_\Omega |\nabla u|^2\di x}{\displaystyle \Big(\int_\Omega \int_\Omega \frac{|u(x)|^{\Ds}|u(y)|^{\Ds}}{|x-y|^\mu}\di x \di y\Big)^{\frac{N-2}{2N-\mu}}}=S_{H,L}.
\end{equation}
Moreover $S_{H,L}^\Omega$ is never achieved unless $\Omega =\setR$.
\end{lemma}

\section{Variational setting}\label{sec3}
To study Problem $\eqref{P}$, we consider the associated functional $I_\eps:H_0^1(\Omega) \rightarrow~\mathbb{R}$ given by
\begin{equation}\label{F}
I_\eps(u):=\dfrac{1}{2}\int_\Omega |\nabla u|^2\di x+\dfrac{\lambda}{2}\int_\Omega u^2 \di x-\dfrac{1}{2p_\eps}\int_\Omega \int_\Omega \dfrac{(u^+(x))^{p_\eps}(u^+(y))^{p_\eps}}{|x-y|^\mu}\di x \di y,
\end{equation}
where $u^+(x):=\max \{u(x),0 \}$ is the positive part of the function $u$.
\begin{prop}\label{problemiequivalenti}
The critical points of $I_\eps$ coincide with the solutions of $\eqref{P}$.
\end{prop}
\begin{proof}
It is easy to see that critical points of $\eqref{P}$ are solutions of
\begin{equation}\label{P+}
\begin{cases}
-\Delta u +\lambda u=\Big(\displaystyle\int_\Omega\dfrac{(u^+)^{p_\eps}(x)}{|x-y|^\mu}\di x \Big)(u^+)^{p_\eps -2}u^+\qquad &\text{in}\; \Omega\\
u=0\qquad &\text{in}\; \mathbb{R}^N \setminus \Omega.
\end{cases}
\end{equation}

We claim that $\eqref{P+}$ is equivalent to $\eqref{P}$. Indeed, if $u\in H_0^1(\Omega)$ is a solution of $\eqref{P}$, $u>0$ in $\Omega$, so $u=u^+$. Therefore $u$ is a solution of $\eqref{P+}$.

On the other hand, if $u$ is a critical point for $I_\eps$, 
\begin{equation}
\begin{aligned}
0=I_\eps'(u)[u^-]&=\int_\Omega |\nabla u^-|^2 \di x+\lambda \int_\Omega |u^-|^2\di x\\
&+\int_\Omega \int_\Omega\frac{(u^+(x))^{p_\eps}(u^+(y))^{p_\eps-2}u^+(y)u^-(y)}{|x-y|^\mu}\di x \di y=0.
\end{aligned}
\end{equation}
Since 
\[
\int_\Omega \int_\Omega\frac{(u^+(x))^{p_\eps}(u^+(y))^{p_\eps-2}u^+(y)u^-(y)}{|x-y|^\mu}\di x \di y=0,
\]
we get that $u^-=0$ a.e.\ from which $u=u^+\ge 0$. Then, the Maximum Principle ensures that $u>0$ in $\Omega$ and hence $u$ solves $\eqref{P}$.
\end{proof}

Note that the functional $I_\eps$ is unbounded from above and from below. Indeed if we compute $I_\eps(tu)$ we have $I_\eps(tu)\rightarrow -\infty$ as $t\rightarrow +\infty$ because $p_\eps>1$, and if we evaluate $I_\eps$ on $u_n(x):=\sin (nx)$ for $n\in \mathbb{N}$ and $x\in \Omega$ we obtain $I_\eps(u_n)\rightarrow +\infty$ as $n\rightarrow +\infty$.
Hence we want to restrict our functional $I_\eps$ to a suitable manifold $\mathcal{N}$ so that $I_\eps$ is bounded from below on $\mathcal{N}$. We define the Nehari manifold associated to functional $\eqref{F}$ as
\begin{equation*}
\mathcal{N}_\eps:=\{u\in H_0^1(\Omega)\setminus \{0\} :G_\eps(u)=0\},
\end{equation*}
where
\begin{equation}
G_\eps(u):=I_\eps'(u)[u]=\int_\Omega(|\nabla u|^2+\lambda u^2)\di x-\int_\Omega \int_\Omega \dfrac{(u^+(x))^{p_\eps}(u^+(y))^{p_\eps}}{|x-y|^\mu}\di x \di y.
\end{equation}
It is easy to see that on $\mathcal{N}_\eps$ the functional $\eqref{F}$ becomes
\begin{equation}\label{IsuNeps}
I_{\eps|_{\mathcal{N}_\eps}}(u)=\dfrac{p_\eps-1}{2p_\eps}\|u\|_\lambda^2 \ge0.
\end{equation}

We recall some properties of the Nehari manifold.
\begin{lemma}\label{condizioninehari}
We have
\begin{enumerate}
\item[i)]$\mathcal{N}_\eps$ is a $C^1$ manifold;
\item[ii)]there exists $c>0$ such that for every $u\in \mathcal{N}_\eps$, it results $c\le \|u\|$;
\item[iii)] for every $u\neq 0$ there exists a unique $t_\eps(u)>0$ such that $t_\eps(u)u\in \mathcal{N}_\eps$;
\item[iv)] 
Set 
\[
m_\eps:=\inf_{u\in \mathcal{N}_\eps}I_\eps(u).
\]
We have that $m_\eps>0$ and
the following equalities are true
\[
m_\eps=\inf_{u\neq 0} \max_{t_\eps>0} I_\eps(t_\eps u)=\inf_{g\in \Gamma_\eps}\max_{t_\eps\in[0,1]} I_\eps(g(t_\eps)),
\]
where
\[
\Gamma_\eps:=\{g\in C([0,1],H_0^1(\Omega)):g(0)=0, I_\eps(g(1))\le 0, g(1) \neq 0 \}.
\]
\end{enumerate}
Moreover $m_\eps$ is achieved by a function $u_\eps \in \mathcal{N}_\eps$, i.e.\
\[
m_\eps=\min_{u\in \mathcal{N}_\eps}I_\eps(u)=I_\eps(u_\eps).
\]
\end{lemma}

\begin{proof}
$(i)$ and $(iv)$ are standard.

As it concerns $(ii)$, for all $u\in \mathcal{N}_\eps$, we have that 
\begin{equation}
\|u\|_\lambda^2=\int_\Omega \int_\Omega \dfrac{(u^+(x))^{p_\eps}(u^+(y))^{p_\eps}}{|x-y|^\mu}\di x \di y,
\end{equation}
from which
\begin{equation}\label{3.4}
\|u\|^2\le c\|u\|_\lambda^2= c \int_\Omega \int_\Omega \dfrac{(u^+(x))^{p_\eps}(u^+(y))^{p_\eps}}{|x-y|^\mu}\di x \di y\le c |u|_{\Peps}^{2\Peps}\le c\|u\|^{2\Peps},
\end{equation}
where we used $\eqref{HLS}$, the equivalence of the norm $\|\cdot \|$ and $\| \cdot\|_\lambda$ and the Sobolev embedding. 
Then,
\[
\|u\|^{2\Peps-2}\ge \frac{1}{c}
\]
from which it results $(ii)$.

Finally we prove $(iii)$. Suppose that there exist $t_{\eps,1}(u)$, $t_{\eps,2}(u)>0$, $t_{\eps,1}(u)\neq t_{\eps,2}(u)$, such that $t_{\eps,1}(u) u \in \mathcal{N}_\eps$ and $t_{\eps,2}(u) u \in \mathcal{N}_\eps$. We have
\begin{equation}\label{3.5}
t_{\eps,1}^2(u)\|u\|_\lambda^2=t_{\eps,1}^{2\Peps}(u)\int_\Omega \int_\Omega \dfrac{(u^+(x))^{p_\eps}(u^+(y))^{p_\eps}}{|x-y|^\mu}\di x \di y.
\end{equation}
and
\begin{equation}\label{3.6}
t_{\eps,2}^2(u)\|u\|_\lambda^2=t_{\eps,2}^{2\Peps}(u)\int_\Omega \int_\Omega \dfrac{(u^+(x))^{p_\eps}(u^+(y))^{p_\eps}}{|x-y|^\mu}\di x \di y.
\end{equation}
Combining $\eqref{3.5}$ and $\eqref{3.6}$ we get
\[
t_{\eps,1}^{2-2\Peps}(u)\|u\|_\lambda^2=t_{\eps,2}^{2-2\Peps}(u)\|u\|_\lambda^2.
\]
and hence $t_{\eps,1}(u)=t_{\eps,2}(u)$.
\end{proof}

We have also the following result, whose proof is standard and will be omitted (see \cite{c} for a complete proof).

\begin{lemma}\label{Pswellbehaves}
The Nehari manifold $\mathcal{N}_\eps$ is a natural constraint for $I_\eps$.

Then, if $\{u_n\}_{n\in \mathbb{N}}\subset \Neps$ is a $(PS)$-sequence for $I_{\eps|_{\mathcal{N}_\eps}}$, it is a $(PS)$-sequence for the free functional $I_\eps$ on $H_0^1(\Omega)$.
\end{lemma}

From now on we will deal with the restricted functional on the Nehari manifold. For the sake of simplicity we will write $I_\eps$ instead of $I_{\eps|_{\mathcal{N}_\eps}}$.

%\begin{remark}\label{Oss1}
%Thanks to Lemma \ref{condizioninehari}, for all $u\in %\mathcal{N}_\eps$ we have that
%\[
%\|u\|\ge c>0.
%\]
%Then, from $\eqref{3.4}$ it follows that
%\[
%0<c<\int_\Omega \int_\Omega \dfrac{(u^+(x))^{p_\eps}(u^+(y))^{p_\eps}}{|x-y|^\mu}\di x \di y.
%\]
%So $\{|u_\eps|_{\Peps} \}_{\eps>0}$ is bounded away from %zero. Moreover by H\"{o}lder inequality,
%\[
%|u_\eps|_{\Peps}\le |\Omega|^{\frac{\Ds-\Peps}{\Ds \cdot \Peps}}|u_\eps|_{\Ds}
%\]
%where $|\Omega|$ is the Lebesgue measure of $\Omega$. So $%\{|u_\eps|_{\Ds} \}_{\eps>0}$ is bounded away from zero.
%\end{remark}

%\begin{remark}
%Note that the previous claims hold also for $\lambda=0$.
%\end{remark}
\section{Two limit problems}\label{sec4}

The key result of this section is Proposition \ref{limitemeps} in which we compute $\lim_{\eps\rightarrow0}m_\eps$ by means of two limit cases of Problem $\eqref{P}$, 
which are here introduced.
First, we study the critical problem in $\setR$
\begin{equation}\label{PLRN}
\begin{cases}
-\Delta u=\Big(\displaystyle\int_\Omega\dfrac{(u)^{\Ds}(x)}{|x-y|^\mu}\di x \Big)(u)^{\Ds-2}u\qquad &\text{in}\; \setR \\
u>0\qquad &\text{in}\; \setR
\end{cases}
\end{equation}
and the corresponding functional $I_*:\mathcal{D}^{1,2}(\setR)\rightarrow \mathbb{R}$ defined as
\begin{equation}\label{LRN}
I_*(u):=\dfrac{1}{2}\int_\setR |\nabla u|^2\di x-\dfrac{1}{2 \cdot \Ds}\int_\setR \int_\setR \dfrac{(u^+(x))^{\Ds}(u^+(y))^{\Ds}}{|x-y|^\mu}\di x \di y.
\end{equation}
We call $\mathcal{N}_*$ the Nehari manifold associated to $I_*$ given by
\begin{equation}
\begin{aligned}
\mathcal{N}_*&:=\{u\in \mathcal{D}^{1,2}(\setR): (I_*)' (u)[u]=0 \}\\
&=\left \{u\in \mathcal{D}^{1,2}(\setR): \int_{\mathbb{R}^N}|\nabla u|^2 \di x=\int_\setR \int_\setR \dfrac{(u^+(x))^{\Ds}(u^+(y))^{\Ds}}{|x-y|^\mu}\di x \di y\right \}.
\end{aligned}
\end{equation}
We compute the value of $m_*:=\inf_{\mathcal{N}_*}I_*$ in the following lemma:

\begin{lemma}\label{mstarvalue}
It results
\[
m_*=\Big(\dfrac{\Ds-1}{2\cdot \Ds}\Big)S_{H,L}^{\frac{\Ds}{\Ds-1}},
\]
where $S_{H,L}$ (defined in $\eqref{SHL}$) is the best constant. Moreover  $m_*$ is achieved by functions $U_{R,a}$ (defined in $\eqref{eqHLS}$).
\end{lemma}

\begin{proof}
First of all we prove that if $A$, $B>0$, we have
\begin{equation*}
\max_{t>0} \Big\{\dfrac{t^2}{2} A-\dfrac{t^{2\cdot \Ds}}{2\cdot \Ds}B\Big\}=\Big(\dfrac{\Ds-1}{2\cdot \Ds}\Big)\Biggl( \dfrac{A^{\Ds}}{B}\Biggr)^{\frac{1}{\Ds-1}}.
\end{equation*}

Then,
\begin{equation}
m_*=\inf_{u\neq 0} \max_{t_*>0} I_*(t_*u)\ge \Big(\dfrac{\Ds-1}{2\cdot \Ds}\Big)S_{H,L}^{\frac{\Ds}{\Ds-1}}.
\end{equation}
Additionally, since functions $U_{R,a}$ satisfy $\eqref{PLRN}$, $U_{R,a} \in \mathcal{N}_*$, and we get
\[
m_*\le I_*(U_{R,a})=\Big(\frac{1}{2}-\frac{1}{2\cdot \Ds} \Big)\int_{\mathbb{R}^N}|\nabla U_{R,a}|^2\di x=\Big(\dfrac{\Ds-1}{2\cdot \Ds}\Big)S_{H,L}^{\frac{\Ds}{\Ds-1}}.
\]
\end{proof}
The minimizers for $m_*$, i.e.\ the bubble functions, will be the model functions to construct approximating sequences for $\{m_\eps\}_{\eps>0}$.
%We call
%\begin{equation}\label{H}
%h:\Omega_r^+\rightarrow \Omega_r^-
%\end{equation}
%the homotopic equivalence map to the identity map.

For $R>1$ and $x_0\in \Omega_r^-$ we define
\begin{equation}\label{ur}
u_{R,x_0}(\cdot):=R^{\frac{N-2}{2}}U_1(R(\cdot-x_0))\chi_{B_r(x_0)}(\cdot)=U_{R,x_0}(\cdot)\chi_{B_r(x_0)}(\cdot),
\end{equation}
where $U_1$ is the standard bubble function (defined in $\eqref{eqHLS}$ with $R=1$ and $a=0$) and $\chi_{B_r(x_0)}$ is the cut-off function defined as
\begin{equation}\label{cutoff}\chi_{B_r(x_0)}(x):=
\begin{cases}
0\qquad &\text{if}\; |x-x_0|>r\\
1\qquad &\text{if}\; |x-x_0|<\frac{r}{2}.
\end{cases}
\end{equation}
Notice that if $x_0 \in  \Omega_r^-$, then $u_{R,x_0} \in H_0^1(\Omega)$.

It is easy to see that $u_R$ satisfies
\begin{equation}\label{gradconserva}
\int_{\Omega}|\nabla u_{R,x_0}|_2^2\di x= \int_{\setR}|\nabla U_{1}|_2^2\di x +o_R(1)
\end{equation}
\begin{equation}\label{intdoppioconserva}
\int_{\Omega} \int_{\Omega} \dfrac{|u_{R,x_0}(x)|^{\Ds}|u_{R,x_0}(y)|^{\Ds}}{|x-y|^\mu}\di x \di y = \int_{\setR}\int_{\setR} \dfrac{|U_1(x)|^{\Ds}|U_1(y)|^{\Ds}}{|x-y|^\mu}\di x \di y+o_R(1),
\end{equation}
\begin{equation}\label{norma2zero}
|u_{R,x_0}|^2_{L^2(\Omega)}=\frac1{R^2}|u_{R,x_0}|^2_{L^2(\setR)}+o_R(1)=o_R(1).
\end{equation}
Here $o_R(1)\rightarrow 0$ as $R\rightarrow +\infty$.

\begin{lemma}\label{m*sup}
It holds
\begin{equation*}
\limsup_{\eps \rightarrow 0} m_\eps \le m_*.
\end{equation*}
\end{lemma}

\begin{proof}
Let $u_{R,x_0}$ be the function defined in $\eqref{ur}$. For any $\eps>0$, consider the unique $t_\eps (u_{R,x_0})>0$ such that 
$t_\eps (u_{R,x_0}) u_{R,x_0}\in \mathcal{N}_\eps$. Then
\begin{equation}\label{4B}
\|t_\eps (u_{R,x_0}) u_{R,x_0}\|_\lambda^2=t_\eps^{2\Peps} (u_{R,x_0})\int_{\Omega} \int_{\Omega} \dfrac{|u_{R,x_0}(x)|^{\Peps}|u_{R,x_0}(y)|^{\Peps}}{|x-y|^\mu}\di x \di y,
\end{equation}
from which
\begin{equation}\label{t1}
t_\eps^{2\Peps-2} (u_{R,x_0})=\frac{\|u_{R,x_0}\|_\lambda^2}{\displaystyle \int_{\Omega}\int_{\Omega}\dfrac{|u_{R,x_0}(x)|^{p_\eps}|u_{R,x_0}(y)|^{p_\eps}}{|x-y|^\mu}\di x \di y}.
\end{equation}
Recalling $\eqref{gradconserva}$, $\eqref{norma2zero}$ and $\eqref{intdoppioconserva}$, since the map $\Peps\mapsto \displaystyle \int_{\Omega}\int_{\Omega}\dfrac{|u_{R,x_0}(x)|^{\Peps}|u_{R,x_0}(y)|^{\Peps}}{|x-y|^\mu}\di x \di y$
is continuous, passing to the limit in $\eqref{t1}$, we obtain
\begin{equation}\label{4D}
\lim_{\eps \rightarrow 0} t_\eps (u_{R,x_0})=\left(
\frac{|\nabla U_1|_{L^2(\setR)}^2+o_R(1)}
{\displaystyle \int_{\mathbb{R}^N}\int_{\mathbb{R}^N}\dfrac{|U_1(x)|^{2_\mu^*}|U_1(y)|^{2_\mu^*}}{|x-y|^\mu}\di x \di y+o_R(1)}
\right)^\frac1{2\Peps-2}=1+o_R(1).
\end{equation}
Therefore
\begin{equation}
I_\eps(t_\eps (u_{R,x_0}) u_{R,x_0})=\dfrac{\Peps-1}{2\Peps}\|t_\eps (u_{R,x_0}) u_{R,x_0}\|_\lambda^2=t_\eps^2 (u_{R,x_0})\dfrac{\Peps-1}{2\Peps}|\nabla U_1|_{L^2(\setR)}^2+o_R(1)
\end{equation}
and passing to the limit for $\eps \rightarrow 0$, in light of $\eqref{4D}$, we have
\begin{equation}\label{4.14}
\lim_{\eps \rightarrow 0} I_\eps(t_\eps (u_{R,x_0}) u_{R,x_0})=\dfrac{\Ds-1}{2\cdot \Ds}|\nabla U_1|_{L^2(\setR)}^2+o_R(1).
\end{equation}
For $\delta>0$ small, we choose $R$ such that $o_R(1)<\delta$,
using $\eqref{4.14}$  and Lemma \ref{mstarvalue} we get
\begin{equation}\label{I_*<m*}
\limsup_{\eps \rightarrow 0}m_\eps \le \lim_{\eps \rightarrow 0} I_\eps(t_\eps (u_{R,x_0}) u_{R,x_0})<
\dfrac{\Ds-1}{2\cdot \Ds}|\nabla U_1|_{L^2(\setR)}^2+\delta=m_*+\delta.
\end{equation}
This completes the proof since $\delta$ is arbitrary.
\end{proof}

\begin{remark}\label{graundbdd}
Observe that, since $\{m_\eps\}_{\eps>0}$ is uniformly bounded in $\eps$, the groundstates are bounded too. Infact
\begin{equation}
\|u_\eps\|_\lambda^2=\frac{2\Peps}{\Peps-1}I_\eps(u_\eps)=\frac{2\Peps}{\Peps-1}m_\eps.
\end{equation}
\end{remark}

Now we introduce a limit problem which acts as mediator between Problem $\eqref{P}$ and Problem $\eqref{PLRN}$. It will have a crucial role in the computation of the limit of $m_\eps$. We consider
\begin{equation}\label{Prob}
\begin{cases}
-\Delta u+ \lambda u=\Big(\displaystyle\int_\Omega  \frac{(u(x))^{\Ds}}{|x-y|^\mu}\di x\Big)(u)^{\Ds-2}u\quad &\text{ in }\;\Omega\\
u>0\quad &\text{ in }\; \Omega ;\\
u=0\quad &\text{ in }\; \mathbb{R}^N \setminus \Omega;
\end{cases}
\end{equation}
its solutions are critical points of the functional $I_*^\Omega:H_0^1(\Omega)\rightarrow \mathbb{R}$ defined as
\begin{equation}\label{L}
I_*^\Omega(u):=\dfrac{1}{2}\int_\Omega |\nabla u|^2\di x+\frac{\lambda}{2}\int_\Omega |u|^2 \di x-\dfrac{1}{2 \cdot \Ds}\int_\Omega \int_\Omega \dfrac{(u^+(x))^{\Ds}(u^+(y))^{\Ds}}{|x-y|^\mu}\di x \di y.
\end{equation}

As usual, the Nehari manifold associated to functional $\eqref{L}$ is
\[
\mathcal{N}_*^\Omega:= \{u\in H_0^1(\Omega) \setminus \{0\}:G_*^\Omega(u)=0\}
\]
where
\[
G_*^\Omega(u):=(I_*^\Omega(u))'[u]=\|u\|_\lambda^2-\int_\Omega \int_\Omega \dfrac{(u^+(x))^{\Ds}(u^+(y))^{\Ds}}{|x-y|^\mu}\di x \di y.
\]
For $u\in \mathcal{N}_*^\Omega$ we have that 
\begin{equation}\label{functonnehariman}
I_*^\Omega(u)=\Big(\frac{1}{2}-\frac{1}{2\cdot \Ds}\Big) \|u\|_\lambda^2
\end{equation}
and $m_*^\Omega :=\inf_{\mathcal{N}_*^\Omega}I_*^\Omega$.
\begin{lemma}\label{m*nonraggiunto}
Note that $m_*^\Omega=m_*$ and  $m_*^\Omega$ is not achieved.
\end{lemma}
\begin{proof}
Obviously, $m_*\le m_*^\Omega$. Moreover, for all $x_0\in \Omega_r^-$ and $R>0$, we take $u_{R,x_0}$ defined as in  $\eqref{ur}$ and $t_*^\Omega(u_{R,x_0})>0$ be the unique value such that $t_*^\Omega(u_{R,x_0})u_{R,x_0}\in \mathcal{N}_*^\Omega$. Proceeding as in Lemma \ref{m*sup}, we have that for every $\delta>0$ there exists $R>0$ such that
\[
m_*^\Omega \le I_*^\Omega(t_*^\Omega(u_{R,x_0})u_{R,x_0})<m_*+\delta.
\]
For the arbitrariness of $\delta$ we get $m_*^\Omega \le m_*$. Hence $m_*=m_*^\Omega$. 

We show that $m_*^\Omega$ is never a minimum. Indeed, suppose by contradiction that $v\in~H_0^1(\Omega)$ is such that $I_*^\Omega(v)=m_*^\Omega$. Then
\begin{equation}
I_*^\Omega(v)=\Big(\frac{1}{2}-\frac{1}{2\cdot \Ds} \Big)\int_\Omega \int_\Omega \dfrac{(v^+(x))^{\Ds}(v^+(y))^{\Ds}}{|x-y|^\mu}\di x \di y=m_*^\Omega.
\end{equation}

We extend $v$ to zero outside $\Omega$ and we call this extension $\bar{v}$. Let $t_*(\bar{v})>0$ the unique value such that $t_*(\bar{v})\bar{v}\in \mathcal{N}_*$. Since $v\in \mathcal{N}_*^\Omega$, it results
\begin{equation}
\begin{aligned}
&t_*^2(\bar{v})\|\bar{v}\|_{\mathcal{D}^{1,2}(\mathbb{R}^N)}^2=t_*^{2\cdot \Ds}(\bar{v})\int_\setR \int_\setR \dfrac{(\bar{v}^+(x))^{\Ds}(\bar{v}^+(y))^{\Ds}}{|x-y|^\mu}\di x \di y\\
&=t_*^{2\cdot \Ds}(\bar{v})\int_\Omega \int_\Omega \dfrac{(v^+(x))^{\Ds}(v^+(y))^{\Ds}}{|x-y|^\mu}\di x \di y=t_*^{2\cdot \Ds}(\bar{v})\|v\|_\lambda^2\\
&=t_*^{2\cdot \Ds}(\bar{v})\|\bar{v}\|_\lambda^2,
\end{aligned}
\end{equation}
from which, if $\lambda>0$, we get
\begin{equation}
t_*(\bar{v})=\Biggr( \frac{\|\bar{v}\|_{\mathcal{D}^{1,2}(\mathbb{R}^N)}^2}{\|\bar{v}\|_\lambda^2}\Biggr)^{1/(2\cdot \Ds -2)}<1.
\end{equation}
Hence, computing
\begin{equation}
\begin{aligned}
I_*(t_*(\bar{v})\bar{v})&=\Big( \frac{1}{2}-\frac{1}{2\cdot \Ds}\Big)t_*^{2\cdot \Ds}(\bar{v})\int_\setR \int_\setR \dfrac{(\bar{v}^+(x))^{\Ds}(\bar{v}^+(y))^{\Ds}}{|x-y|^\mu}\di x \di y\\
&=\Big( \frac{1}{2}-\frac{1}{2\cdot \Ds}\Big)t_*^{2\cdot \Ds}(\bar{v})\int_\Omega \int_\Omega \dfrac{(v^+(x))^{\Ds}(v^+(y))^{\Ds}}{|x-y|^\mu}\di x \di y\\
&<\Big( \frac{1}{2}-\frac{1}{2\cdot \Ds}\Big)\int_\Omega \int_\Omega \dfrac{(v^+(x))^{\Ds}(v^+(y))^{\Ds}}{|x-y|^\mu}\di x \di y=m_*^\Omega,
\end{aligned}
\end{equation}
we have a contradiction.

On the other hand, if $\lambda=0$, since $v\in \mathcal{N}_*^\Omega$ we get that $t_*(\bar{v})=1$, so $\bar{v}\in \mathcal{N}_*$. We recall from Proposition \ref{problemiequivalenti} that $v\ge 0$, hence $\bar{v}\ge 0$. Moreover $I_*(\bar{v})=m_*^\Omega=m_*=\min_{\mathcal{N}_*}I_*$. 
Thus $\bar{v}$ satisfies Problem $\eqref{Prob}$. By the Maximum Principle it follows that $\bar{v}>0$, but from its construction it is not.
\end{proof}

Now we prove the main result of this section:

\begin{prop}\label{limitemeps}
It holds
\[
\lim_{\eps \rightarrow 0} m_\eps=m_*.
\]
\end{prop}

\begin{proof}
By Lemma \ref{m*sup} it is sufficient to show that
\begin{equation}\label{parte1}
m_*\le \liminf_{\eps \rightarrow 0}m_\eps.
\end{equation}
Let $u_\eps \in \mathcal{N}_\eps$ such that $I_\eps(u_\eps)=m_\eps$. Let $t_*^\Omega(u_\eps)>0$ the unique value such that $t_*^\Omega(u_\eps)u_\eps \in~\mathcal{N}_*^\Omega$, that is
\begin{equation}
(t_*^\Omega(u_\eps))^2\|u_\eps\|_\lambda^2=(t_*^\Omega(u_\eps))^{2\cdot \Ds}\int_\Omega \int_\Omega \frac{(u_\eps^+(x))^{\Ds}(u_\eps^+(y))^{\Ds}}{|x-y|^{\mu}}\di x \di y.
\end{equation}
Recalling that $u_\eps \in \mathcal{N}_\eps$, we obtain
\begin{equation}\label{ttende1}
\begin{aligned}
(t_*^\Omega(u_\eps))^{2\cdot \Ds-2}&=\frac{\|u_\eps\|_\lambda^2}{\displaystyle \int_\Omega \int_\Omega \frac{(u_\eps^+(x))^{\Ds}(u_\eps^+(y))^{\Ds}}{|x-y|^{\mu}}\di x \di y}=\frac{\displaystyle \int_\Omega \int_\Omega \frac{(u_\eps^+(x))^{\Peps}(u_\eps^+(y))^{\Peps}}{|x-y|^{\mu}}\di x \di y}{\displaystyle \int_\Omega \int_\Omega \frac{(u_\eps^+(x))^{\Ds}(u_\eps^+(y))^{\Ds}}{|x-y|^{\mu}}\di x \di y}.
\end{aligned}
\end{equation}
We claim that 
\begin{equation}\label{temposotto1}
\limsup_{\eps \rightarrow 0}(t_*^\Omega(u_\eps))^{2\cdot \Ds-2}\le 1.
\end{equation}
 To see this, we write
\begin{equation*}
\int_\Omega \int_\Omega \frac{(u_\eps^+(x))^{\Peps}(u_\eps^+(y))^{\Peps}}{|x-y|^{\mu}}\di x \di y=\int_\Omega \int_\Omega \frac{(u_\eps^+(x))^{\Peps}(u_\eps^+(y))^{\Peps}}{|x-y|^{\mu \cdot \frac{\Peps}{\Ds}}}\cdot \frac{1}{|x-y|^{\mu (1-\Peps/\Ds)}}\di x \di y
\end{equation*}
and by H\"{o}lder inequality we have
\begin{equation}
\begin{aligned}
\int_\Omega \int_\Omega &\frac{(u_\eps^+(x))^{\Peps}(u_\eps^+(y))^{\Peps}}{|x-y|^{\mu}}\di x \di y\\
&\le \Big( \int_\Omega \int_\Omega \frac{(u_\eps^+(x))^{\Ds}(u_\eps^+(y))^{\Ds}}{|x-y|^{\mu}}\di x \di y\Big)^{\frac{\Peps}{\Ds}}\cdot \Big(\int_\Omega \int_\Omega \frac{1}{|x-y|^\mu}\di x \di y \Big)^{\frac{\Ds-\Peps}{\Ds}}. 
\end{aligned}
\end{equation}
By the change of variables $\xi =x-y$, $\eta=x+y$, for $\rho=\rho(\Omega)>0$ sufficiently large, we get that
\begin{equation}
\int_\Omega \int_\Omega \frac{1}{|x-y|^\mu}\di x \di y\le \int_{B_\rho}\int_{B_\rho}\frac{1}{|\xi|^\mu} \di \xi \di \eta \le C(\rho) \int_{B_\rho}\frac{1}{|\xi|^\mu}\di \xi=C(\rho),
\end{equation}
being $\mu<N$. Since $\rho$ depends only on $\Omega$ we obtain
\begin{equation}\label{holderdoppia}
\int_\Omega \int_\Omega \frac{(u_\eps^+(x))^{\Peps}(u_\eps^+(y))^{\Peps}}{|x-y|^{\mu}}\di x \di y\le \Big( \int_\Omega \int_\Omega \frac{(u_\eps^+(x))^{\Ds}(u_\eps^+(y))^{\Ds}}{|x-y|^{\mu}}\di x \di y\Big)^{\frac{\Peps}{\Ds}}\cdot(C(\Omega))^{\frac{\eps}{\Ds}}.
\end{equation}
Then, putting $\eqref{holderdoppia}$ in $\eqref{ttende1}$ we have
\begin{equation}\label{429}
\begin{aligned}
(t_*^\Omega(u_\eps))^{2\cdot \Ds-2}&\le \Big( \int_\Omega \int_\Omega \frac{(u_\eps^+(x))^{\Ds}(u_\eps^+(y))^{\Ds}}{|x-y|^{\mu}}\di x \di y\Big)^{\frac{\Peps}{\Ds}-1}\cdot(C(\Omega))^{\frac{\eps}{\Ds}}\\
&=\left (\frac{C(\Omega)}{\displaystyle \int_\Omega \int_\Omega \frac{(u_\eps^+(x))^{\Ds}(u_\eps^+(y))^{\Ds}}{|x-y|^{\mu}}\di x \di y}\right )^{\frac{\eps}{\Ds}}.
\end{aligned}
\end{equation}
Since $u_\eps \in \mathcal{N}_\eps$ and Proposition \ref{HLS} holds, $\displaystyle \int_\Omega \int_\Omega \frac{(u_\eps^+(x))^{\Ds}(u_\eps^+(y))^{\Ds}}{|x-y|^{\mu}}\di x \di y$ is uniformly bounded in $\eps$, and hence from $\eqref{429}$ we deduce $\eqref{temposotto1}$.

Thus 
\begin{equation*}
\begin{aligned}
m_*&=m_*^\Omega \le I_*^\Omega (t_*^\Omega(u_\eps)u_\eps)=(t_*^\Omega(u_\eps))^{2\cdot \Ds}\Big(\frac{1}{2}-\frac{1}{2\cdot \Ds}\Big) \int_\Omega \int_\Omega \frac{(u_\eps^+(x))^{\Ds}(u_\eps^+(y))^{\Ds}}{|x-y|^{\mu}}\di x \di y\\
&=(t_*^\Omega(u_\eps))^{2\cdot \Ds}\Big(\frac{1}{2}-\frac{1}{2\cdot \Ds}\Big)\|u_\eps\|_\lambda^2=(t_*^\Omega(u_\eps))^{2\cdot \Ds}\frac{\frac{1}{2}-\frac{1}{2\cdot \Ds}}{\frac{1}{2}-\frac{1}{2\Peps}}\cdot \Big(\frac{1}{2}-\frac{1}{2\Peps}\Big)\|u_\eps\|_\lambda^2\\
&\le (1+o(1))m_\eps,
\end{aligned}
\end{equation*}
where $o(1)\rightarrow 0$ as $\eps\rightarrow 0$ and $m_\eps$ is bounded from Lemma \ref{m*sup}. Therefore $\eqref{parte1}$ is showed.
\end{proof}
\section{The barycenter map}\label{sec5}
We introduce the barycenter of a function $u\in H^1(\setR)$ as
\begin{equation}\label{eq:bari}
\beta(u):=\dfrac{\int_{\mathbb{R}^N}x^i|\nabla u|^2\di x}{\int_{\mathbb{R}^N}|\nabla u|^2\di x}.
\end{equation}

Now we state a {\em splitting lemma} which gives us a complete description of the functional $I_*^\Omega$ (defined in $\eqref{L}$):
\begin{thm}\label{conccmptnes}
Let $\Omega$ be a regular and bounded domain of $\mathbb{R}^N$, $N> 3$ and let $\{v_n \}_{n\in \mathbb{N}}$ a $(PS)$ sequence for $I_*^\Omega$ in $H_0^1(\Omega)$.

Then there exist $k\in \mathbb{N}_0$, a sequence $\{x_n^j\}_{n\in \mathbb{N}}$ of points $x_n^j\in \Omega$, a sequence $\{R_n^j\}_{n\in \mathbb{N}}$ of radii $R_n^j\rightarrow +\infty$ as $n\rightarrow+\infty$ ($1\le j\le k$), a solution $v\in H_0^1(\Omega)$ of $\eqref{Prob}$ and non-trivial solutions $v^j\in \mathcal{D}^{1,2}(\mathbb{R}^N)$, $1\le j \le k$, to the limit problem $\eqref{PLRN}$ such that a subsequence $\{v_{n}\}_{n\in \mathbb{N}}$ satisfies
\begin{equation}\label{spiltfnct}
\Big\|v_n-v-\sum_{j=1}^k{v_{R_n,x_n}^j}\Big\|_{\mathcal{D}^{1,2}(\mathbb{R}^N)}\rightarrow 0.
\end{equation}
Here $v_{R_n,x_n}^j$ denotes the rescaled functions
\[
v_{R_n,x_n}^j(x):=(R_n^j)^{\frac{N-2}{2}}v^j(R_n^j(x-x_n^j))\quad 1\le j\le k.
\]
Moreover we have the following splitting
\begin{equation}\label{fnctionalsplit}
I_*^\Omega(v_n)\rightarrow I_*^\Omega(v)+\sum_{j=1}^k{I_*(v^j)}.
\end{equation}
\end{thm}

\begin{remark}\label{dopoteokey}
If we have a $(PS)$-sequence $\{v_n \}_{n\in \mathbb{N}}$ for $I_*^\Omega$ at level $m_*^\Omega$, from Lemma \ref{m*nonraggiunto} we know that $m_*^\Omega$ is not achieved, thus from Theorem \ref{conccmptnes} we have that the only contribution in the r.h.s of $\eqref{fnctionalsplit}$ comes from $U_{R,a}$, i.e.\
\[
v=0,\quad k=1\quad \text{and}\quad v^1=U_{R,a}
\]
and hence
\begin{equation}\label{bubblevicina}
v_n-U_{R_n,x_n}\rightarrow 0 \quad \text{in}\; \mathcal{D}^{1,2}(\mathbb{R}^N).
\end{equation}
\end{remark}
By Theorem \ref{conccmptnes} we can prove the following property:

\begin{prop}\label{baricentrodentro}
There exist $\delta_0>0$ and $\eps_0=\eps_0(\delta_0)>0$ such that for any $\delta \in (0,\delta_0]$ and for any $\eps\in (0,\eps_0]$ it holds 
\[
u\in \mathcal{N}_\eps\quad \text{and}\quad I_\eps(u)<m_\eps +\delta \Rightarrow \beta (u)\in \Omega_r^+.
\]
\end{prop}

\begin{proof}
By contradiction we suppose that there exist sequences $\delta_n \rightarrow 0$, $\eps_n\rightarrow 0$ and $u_n \in \mathcal{N}_{\eps_n}$ such that
\begin{equation}\label{Ha}
I_{\eps_n}(u_n)\le m_{\eps_n} +\delta_n \quad \text{and}\quad \beta(u_n)\notin \Omega_r^+.
\end{equation}
From this inequality and Proposition \ref{limitemeps} we deduce that
\begin{equation}\label{5A}
I_{\eps_n}(u_n)\rightarrow m_*
\end{equation}
and $\{u_n \}_{n\in \mathbb{N}}$ is bounded in $H_0^1(\Omega)$. We take $t_{*,n}^\Omega(u_n)>0$ the unique value such that $t_{*,n}^\Omega(u_n) u_n\in~\mathcal{N}_*^\Omega$. Defined $p_n:=\Ds-\eps_n$, we evaluate
\begin{equation}
\begin{aligned}
I_{\eps_n}(u_n)&- I_*^\Omega(t_{*,n}^\Omega(u_n) u_n)=\Big(\frac{1}{2}-\frac{1}{2p_n}\Big)\|u_n\|_\lambda^2-\Big(\frac{1}{2}-\frac{1}{2\cdot \Ds}\Big)(t_{*,n}^\Omega(u_n))^2\|u_n\|_\lambda^2 \\
&= \Big(\frac{1}{2}-\frac{1}{2p_n}\Big)(1-(t_{*,n}^\Omega(u_n))^2)\|u_n\|_\lambda^2-\Big(\frac{1}{2p_n}-\frac{1}{2\cdot \Ds}\Big)(t_{*,n}^\Omega(u_n))^2\|u_n\|_\lambda^2.
\end{aligned}
\end{equation}
Proceeding as in Proposition \ref{limitemeps}, we have $t_{*,n}^\Omega(u_n) \le 1+o(1)$ where $o(1)\rightarrow 0$ as $n\rightarrow +\infty$, and hence
\[
\Big(\frac{1}{2}-\frac{1}{2p_n}\Big)(1-(t_{*_n}^\Omega(u_n))^2\|u_n\|_\lambda^2\ge o(1)
\]
and since $p_n \rightarrow \Ds$, we get
\[
\Big(\frac{1}{2p_n}-\frac{1}{2\cdot \Ds}\Big)(t_{*,n}^\Omega(u_n))^2)\|u_n\|_\lambda^2=o(1),
\]
from which
\[
I_{\eps_n}(u_n)- I_*^\Omega(t_{*_n}^\Omega(u_n) u_n)\ge o(1).
\]
Then, thanks to $\eqref{5A}$, we obtain
\[
m_*=m_*^\Omega\le \lim_{n \rightarrow +\infty} I_*^\Omega(t_{*,n}^\Omega(u_n) u_n)\le \lim_{n \rightarrow +\infty} \Big( I_{\eps_n}(u_n) +o(1) \Big)=m_*, 
\]
that is
\[
I_*^\Omega(t_{*,n}^\Omega(u_n) u_n)\rightarrow m_*\quad \text{as}\; n\rightarrow +\infty.
\]
Now, Ekeland's variational principle gives us the existence of $\{v_n \}_{n\in \mathbb{N}}\subset \mathcal{N}_*^\Omega$ and $\{\mu_n \}_{n\in \mathbb{N}}\subset~\mathbb{R}$ such that
\begin{equation}\label{EK}
\|t_{*,n}^\Omega(u_n) u_n-v_n\|\rightarrow 0
\end{equation}
\[
I_*^\Omega(v_n)=\Big(\frac{1}{2}-\frac{1}{2\cdot \Ds} \Big)\| v_n\|_\lambda^2\rightarrow m_*
\]
\[
(I_*^\Omega)'(v_n)-\mu_n(G_*^\Omega)'(v_n)\rightarrow 0.
\]
Applying Lemma \ref{Pswellbehaves}, we obtain that $\{v_n \}_{n\in \mathbb{N}}$ is a $(PS)$-sequence also for the free functional $I_*^\Omega$ at level $m_*$, hence Remark \ref{dopoteokey} implies that
\[
v_n-U_{R_n,x_n}\rightarrow 0\quad \text{in}\; \mathcal{D}^{1,2}(\mathbb{R}^N)
\]
where $\{x_n\}_{n\in \mathbb{N}}\subset \Omega$ and $R_n \rightarrow +\infty$.
In an equivalent way, we can write
\[
v_n=U_{R_n,x_n}+w_n,
\]
with $\|w_n\|_{\mathcal{D}^{1,2}(\mathbb{R}^N)}\rightarrow 0$. From $\eqref{EK}$, unless to relabel $w_n$, we have that
\begin{equation}\label{tomsplit}
t_{*,n}^\Omega(u_n) u_n=U_{R_n,x_n}+w_n.
\end{equation}
Using this equality, if $x=(x^1, \cdots,x^N)\in \mathbb{R}^N$, we get that
\begin{equation}\label{B}
\begin{aligned}
\beta(t_{*,n}^\Omega(u_n) u_n)^i\|t_{*,n}^\Omega(u_n) u_n\|_{\mathcal{D}^{1,2}(\mathbb{R}^N)}^2&=\int_{\mathbb{R}^N}x^i|\nabla U_{R_n,x_n}(x)|^2\di x+\int_{\mathbb{R}^N}x^i|\nabla w_n(x)|^2 \di x\\
&+2 \int_{\setR}x^i\nabla U_{R_n,x_n}(x)\nabla w_n(x)\di x=:J_1+J_2+2J_3.
\end{aligned}
\end{equation}

Since $
U_{R_n, x_n}(x)=R_n^{\frac{N-2}{2}}U_1(R_n(x-x_n)),
$
using the change of variables
$
y=R_n(x-x_n),
$
we obtain 
\begin{equation}\label{5C}
\begin{aligned}
J_1&:=\int_{\mathbb{R}^N}x^i|\nabla U_{R_n,x_n}(x)|^2\di x=\int_{\mathbb{R}^N}x^i R_n^N |\nabla U_1(R_n(x-x_n))|^2\di x\\
&=\frac{1}{R_n}\int_{\setR}y^i|\nabla U_1(y)|^2\di y+x_n^i\int_{\setR}|\nabla U_1(y)|^2\di y\\
&=o_{n}(1)+x_n^i\int_{\setR}|\nabla U_1(y)|^2\di y\qquad \text{since}\; N>3,
\end{aligned}
\end{equation}
where $o_n(1)\rightarrow 0$ as $n\rightarrow +\infty$.

Analogously, from $\eqref{tomsplit}$, we obtain
\begin{equation}\label{5B}
\|t_{*_n}^\Omega(u_n) u_n\|_{\Dud}^2=\|U_1\|_{\Dud}^2+o_{R_n}(1).
\end{equation}
 
To evaluate $J_2$, we recall that $\{v_n\}_{n\in \N}$ is a subsequence supported in $\Omega$, so
\[
U_{R_n,x_n}=-w_n\quad \text{in}\; \setR \setminus \Omega
\]
and hence, since $w_n\rightarrow 0$ in $\mathcal{D}^{1,2}(\mathbb{R}^N)$,
\begin{equation}
\begin{aligned}
J_2&:=\int_{\mathbb{R}^N}x^i|\nabla w_n(x)|^2 \di x=\int_\Omega x^i|\nabla w_n(x)|^2 \di x+\int_{\setR \setminus \Omega}x^i|\nabla w_n(x)|^2 \di x\\
&=o_n(1)+\int_{\setR \setminus \Omega}x^iR_n^{N}|\nabla U_1(R_n(x-x_n))|^2\di x\\
&=o_n(1)+\int_{\setR \setminus R_n(\Omega-x_n)}\Big(\frac{y^i}{R_n}+x_n^i\Big)|\nabla U_1(y)|^2\di y\\
&=o_n(1)+\frac{1}{R_n}\int_{\setR \setminus R_n(\Omega-x_n)}y^i|\nabla U_1(y)|^2\di y+x_n^i\int_{\setR \setminus R_n(\Omega-x_n)}|\nabla U_1(y)|^2\di y\\
&=o_{n}(1).
\end{aligned}
\end{equation}
Finally in the same way
\begin{equation}\label{5E}
J_3=o_n(1).
\end{equation}

Putting $\eqref{5B}$, $\eqref{5C}$ and $\eqref{5E}$ in $\eqref{B}$ we obtain
\begin{equation}\label{barinome}
\beta(u_n)^i=\beta(t_{*,n}^\Omega(u_n) u_n)^i=\frac{x_n^i\|U_1\|_{\Dud}^2+o_n(1)}{\|U_1\|_{\Dud}^2+o_n(1)}.
\end{equation}
Since $\{x_n \}_{n\in \mathbb{N}}\subset \Omega$, from $\eqref{barinome}$, we deduce that $\beta(u_n)\in \Omega_r^+$ for $n$ large, which contradicts $\eqref{Ha}$ and concludes the proof.
\end{proof}

At this point we are ready to prove our main theorem:
\begin{proof}[Proof of Theorem $1.1$]
Fix $\delta_0>0$ and $\eps_0(\delta_0)>0$ as in Proposition \ref{baricentrodentro}. Then, for all $\eps<\eps_0(\delta_0)$ it holds
\begin{equation}\label{barri}
u\in \mathcal{N}_\eps\quad \text{and}\quad I_\eps(u)<m_\eps +\delta_0 \Rightarrow \beta (u)\in \Omega_r^+.
\end{equation}
In correspondence of $\delta_0$, by Proposition \ref{limitemeps}, there exists $\bar{\eps}(\delta_0)>0$ such that
\begin{equation}\label{limite}
|m_*-m_\eps|<\frac{\delta_0}{2}\qquad \forall \; \eps<\bar{\eps}(\delta_0).
\end{equation}
Moreover, by Lemma \ref{m*sup}, there exists $\hat{\eps}(\delta_0)>0$ such that for all $\eps<\hat{\eps}(\delta_0)$ there exists $R=R(\delta_0,\eps)>1$ such that
\begin{equation}\label{stosotto}
I_\eps(t_\eps(u_{R,x_0}) u_{R,x_0}(x))\le m_*+\frac{\delta_0}{2},
\end{equation}
where $u_{R,x_0}$ is introduced in $\eqref{ur}$ and $t_\eps(u_{R,x_0})>0$ is the unique value such that $t_\eps(u_{R,x_0})u_{R,x_0}\in \mathcal{N}_\eps$.

Then, taking $\eps<\min \{\eps_0(\delta_0), \bar{\eps}(\delta_0), \hat{\eps}(\delta_0)\}$ and choosing ${R}={R}(\delta_0,\eps)>1$, we define as
\begin{equation}\label{psii}
\begin{aligned}
\psi_{\eps}&:\Omega_r^-\rightarrow \mathcal{N}_\eps\\
\psi_\eps(x)&:=t_\eps(u_{{R},x_0}) u_{{R},x_0}(x).
\end{aligned}
\end{equation}
From $\eqref{psii}$, $\eqref{stosotto}$ and $\eqref{limite}$, we get that
\begin{equation}\label{psibndef}
\psi_\eps(\Omega_r^-)\subset \mathcal{N}_\eps \cap I_\eps^{m_*+\delta_0/2}\subset \mathcal{N}_\eps \cap I_\eps^{m_\eps+\delta_0},
\end{equation}
where, for $c\in \mathbb{R}$, $I_\eps^c:=\{u_\eps \in \mathcal{N}_\eps :I_\eps(u)\le c \}$ denotes the sublevel set.

From $\eqref{psibndef}$ and $\eqref{barri}$, the following maps are well-defined
\[
\Omega_r^-\xrightarrow{\psi_{\eps}} \mathcal{N}_\eps \cap I_\eps^{m_\eps+\delta_0} \xrightarrow{\beta} \Omega_r^+
\]
and $\beta \circ \psi_\eps$ is homotopic to the identity on $\Omega_r^-$. So, thanks to Remark \ref{catinclusione} we get
\begin{equation}\label{categorieincluse}
\text{cat}_{I_\eps^{m_\eps+\delta_0}}(\psi_\eps(\Omega_r^-))\ge \text{cat}_{\Omega_r^-}(\Omega_r^-)
=\text{cat}_{\Omega}(\Omega)>1,
\end{equation}
hence we found a sublevel $I_\eps$ of $\mathcal{N}_\eps$ whose category is greater than cat$_{\Omega}(\Omega)$. Then, we can apply Theorem \ref{LSthm} to obtain the existence of at least cat$_{\Omega}(\Omega)$ critical points in $I_\eps^{m_\eps+\delta_0}$.

It remains to show  that there exists another solution of $\eqref{P}$ such that $I_\eps(\bar{u})>m_\eps+\delta_0$.

To obtain this, proceeding as in \cite{BBM}*{Section 6}, we construct a set $T_\eps$ such that
\[
\psi_\eps(\Omega_r^-) \subset T_\eps \subset \mathcal{N}_\eps \cap I_\eps^{c_\eps}
\]
for some $c_\eps$, $T_\eps$ is contractible in $\mathcal{N}_\eps \cap I_\eps^{c_\eps}$ and it contains only positive functions. Then, by Theorem \ref{LSthm}, we have the claim. We can also show that $c_\eps\le c$ for all $\eps$.

Consider $v\in H_0^1(\mathbb{R}^N)$ any positive function and for $\bar{x}\in \Omega_r^-$ we set 
\[
\bar{v}(\cdot):=v(\cdot)\chi_{B_r(\bar{x})}(\cdot),
\] 
where $\chi_{B_r(\bar{x})}$ is defined as in $\eqref{cutoff}$.

We denote with
\[
C_\eps:=\{\theta \bar{v}(x)+(1-\theta) \Psi(x): \theta\in [0,1],\;\Psi\in \psi_\eps(\Omega_r^-)  \}\subset H_0^1(\Omega).
\]
$C_\eps$ is compact, contractible in $H_0^1(\Omega)$ and it contains only positive functions.

Given $u\in H_0^1(\Omega)$, we take the unique value $t_\eps(u)>0$ such that $t_\eps(u)u \in \mathcal{N}_\eps$, that is
\begin{equation}\label{tepsilon}
(t_\eps(u))^{2\Peps-2}:=\frac{\|u\|_\lambda^2}{\displaystyle \int_\Omega\int_\Omega \frac{|u^+(x)|^{\Peps}|u^+(y)|^{\Peps}}{|x-y|^\mu}\di x\di y}
\end{equation}
and we set
\[
T_\eps:=\{t_\eps(u)u:u \in C_\eps \}\subset \mathcal{N}_\eps.
\]
We see that $\psi_\eps(\Omega_r^-) \subset T_\eps$, $T_\eps$ is compact and contractible in $\mathcal{N}_\eps$ and $T_\eps$ contains only positive functions. 

Finally, denoting with
\[
c_\eps:=\max_{u\in C_\eps} I_\eps(t_\eps(u)u),
\]
we get that $T_\eps \subset \mathcal{N}_\eps \cap I_\eps^{c_\eps}$.
\end{proof}

\begin{lemma}
There exists $c>0$ such that for every $\eps>0$ it results $c_\eps<c$.
\end{lemma}

\begin{proof}
By the definition of $\mathcal{N}_\eps$ we know that if $u\in C_\eps$ we have
\begin{equation}\label{ultima1}
I_\eps(t_\eps(u)u)=\frac{p_\eps-1}{2p_\eps}(t_\eps(u))^2 \|u\|_\lambda^2.
\end{equation}
Now, from the definition of $C_\eps$ and since $u$, $v\in H_0^1(\Omega)$, it follows that
\begin{equation}\label{ultima2}
\|u\|_{\mathcal{D}^{1,2}(\Omega)}\le \|v\|_{\mathcal{D}^{1,2}(\Omega)}+\|u_{R,x_0}\|_{\mathcal{D}^{1,2}(\Omega)}\le c
\end{equation}
and 
\begin{equation}\label{ultima3}
|u|_2^2\le |v|_2^2+|u_{R,x_0}|_2^2\le c.
\end{equation}
Moreover, since the domain $\Omega$ is bounded, we denote with diam $\Omega$ its diameter. For all $(x,y) \in \Omega \times\Omega$ we have $|x-y|\le 2 |\text{diam } \Omega|$, so 
\begin{equation}\label{ultima4}
\begin{aligned}
\int_\Omega\int_\Omega \frac{|u^+(x)|^{\Peps}|u^+(y)|^{\Peps}}{|x-y|^\mu}\di x\di y&\ge \frac{1}{(2|\text{diam }\Omega|)^{\mu}}\int_\Omega\int_\Omega |u^+(x)|^{\Peps}|u^+(y)|^{\Peps}\di x \di y\\
&=\frac{1}{(2|\text{diam }\Omega|)^{\mu}}|u^+|_{\Peps}^{2\Peps}.
\end{aligned}
\end{equation}
We claim that
\begin{equation}\label{Ultimiss}
\frac{1}{(2|\text{diam }\Omega|)^{\mu}}|u^+|_{\Peps}^{2\Peps}>0
\end{equation}
so that, recalling $\eqref{tepsilon}$, we obtain the boundedness of $t_\eps(u)$.
To see this, we note that for $\eps$ small enough, we find constants $c_1>0$, $c_2>0$ such that
\begin{equation}\label{ultima5}
|\bar{v}|^{\Peps}\ge |v|^{\Ds}-c_1>0\qquad |u_{R,x_0}|^{\Peps}\ge |U_1|^{\Ds}-c_2>0.
\end{equation}
Therefore, since $\bar{v}$ and $u_{R,x_0}$ are positive functions and $\theta \in [0,1]$, it results
\begin{equation}\label{ultima6}
|u^+|^{\Peps}\ge \max \{ \theta |\bar{v}|^{\Peps}|, (1-\theta)|u_{R,x_0}|^{\Peps}\} \ge \frac{1}{2}\min \{ |v|^{\Ds}-c_1, |U_1|^{\Ds}-c_2\}>0,
\end{equation}
and the claim is proved.
Putting together $\eqref{ultima1}$, $\eqref{ultima2}$, $\eqref{ultima3}$ and $\eqref{Ultimiss}$ we obtain the thesis.
\end{proof}
\appendix
\section{Proof of Theorem \ref{conccmptnes}}\label{sec7}
For the sake of completeness, we sketch in this appendix the proof of Theorem \ref{conccmptnes}. We follow Lemma $2.3$, Theorem $3.1$ and Lemma $3.3$ of \cite{ST}. 

First of all we recall that a $(PS)$-sequence for $I_*^\Omega$ is bounded. Hence, up to a subsequence, we may assume $v_n \rightharpoonup v$ in $\HO$ and $v$ solves Problem $\eqref{Prob}$. Then, if we consider $w_n=v_n-v$ we have $w_n\rightarrow 0$ in $L^2(\Omega)$ and, by Vitali's convergence
\begin{equation}\label{7C}
\begin{aligned}
\int_\Omega &(|\nabla(v_n)|^2-|\nabla w_n|^2)\di x=\int_\Omega \int_0^1 \frac{\di}{\di t}|\nabla v_n+(t-1)\nabla v|^2\di t \di x\\
&=2\int_0^1\int_\Omega (\nabla v_n+(t-1)\nabla v)\nabla v\di x\di t\rightarrow 2\int_0^1\int_\Omega t|\nabla v |^2\di x \di t\\
&=\int_\Omega |\nabla v|^2\di x
\end{aligned}
\end{equation}
as $n\rightarrow +\infty$. Similarly we obtain 
\begin{equation}\label{7CC}
\int_\Omega (|v_n|^2-|w_n|^2)\di x\rightarrow\int_\Omega |v|^2 \di x
\end{equation}
as $n \rightarrow+\infty$. Then, by \cite{GY}*{Lemma $2.2$} it holds
\begin{equation}\label{7D}
\begin{aligned}
\int_\Omega \int_\Omega\frac{|v_n(x)|^{\Ds}|v_n(y)|^{\Ds}}{|x-y|^\mu}\di x\di  y&=\int_\Omega \int_\Omega\frac{|w_n(x)|^{\Ds}|w_n(y)|^{\Ds}}{|x-y|^\mu}\di x\di y\\
&+\int_\Omega \int_\Omega\frac{|v(x)|^{\Ds}|v(y)|^{\Ds}}{|x-y|^\mu}\di x\di y+o(1)
\end{aligned}
\end{equation}
where $o(1)\rightarrow 0$ as $n\rightarrow +\infty$.

Therefore we obtain
\[
I_*^\Omega(v_n)=I_*^\Omega(v)+I_*^\Omega(w_n)+o(1)
\]
and in the same way
\[
(I_*^\Omega)'(v_n)=(I_*^\Omega)'(v)+(I_*^\Omega)'(w_n)+o(1)=(I_*^\Omega)'(w_n)+o(1)
\]
where $o(1)\rightarrow 0$ as $n\rightarrow +\infty$.

We need the following lemma:

\begin{lemma}\label{lemmaausiliario}\cite{ST}*{Lemma $3.3$}
Let $\{z_n\}_{n\in \N}$ be a $(PS)$-sequence for $I_*^\Omega$ in $\HO$ such that $z_n \rightharpoonup 0$. Then there exist two sequences $\{x_n\}_{n\in \mathbb{N}}\subset \Omega$ of points and $\{R_n\}_{n\in \mathbb{N}}$ of radii, $R_n \rightarrow +\infty$ as $n \rightarrow +\infty$, a non-trivial solution $z$ to the limit problem $\eqref{PLRN}$ and a $(PS)$-sequence $\{w_n\}_{n\in \N}$ for $I_*^\Omega$ in $\HO$ such that for a subsequence $\{z_n\}_{n\in \N}$ there holds
\[
z_n=w_n+R_n^{\frac{N-2}{2}}z(R_n(\cdot-x_n))+o(1)
\]
where $o(1) \rightarrow 0$ as $n\rightarrow +\infty$. In particular $w_n \rightharpoonup0 $ and
\[
I_*^\Omega(w_n)=I_*^\Omega(z_n)-I_*(z)+o(1).
\]
Moreover
\[
R_n\dist (x_n,\partial \Omega)\rightarrow \infty.
\]
Finally if $I_*^\Omega(z_n)\rightarrow m<m_*$, the sequence $\{z_n\}_{n\in \mathbb{N}}$ is relatively compact and hence $z_n\rightarrow 0$ in $H_0^{1}(\Omega)$, $I_*^\Omega(z_n)\rightarrow m=0$ as $n\rightarrow +\infty$.
\end{lemma}

\begin{proof}
We begin noticing that, being $m_*^\Omega:=\inf_{\mathcal{N}_*^\Omega}I_*^\Omega$, we can suppose that $I_*^\Omega(z_n)\rightarrow m \ge m_*^\Omega=\Big(\dfrac{\Ds-1}{2\cdot \Ds}\Big)S_{H,L}^{\frac{\Ds}{\Ds-1}}$.

Then, being $(I_*^\Omega)'(z_n)\rightarrow 0$ we have
\[
\Big(\dfrac{\Ds-1}{2\cdot \Ds}\Big)\|z_n\|_\lambda^2=I_*^\Omega(z_n)-\frac{1}{2\cdot \Ds}\langle z_n, (I_*^\Omega)'(z_n)\rangle \rightarrow m
\]
and hence
\begin{equation}\label{7F}
\liminf_{n\rightarrow +\infty}\|z_n\|_\lambda^2=m \Big(\frac{2\cdot \Ds}{\Ds-1} \Big)\ge S_{H,L}^{\frac{\Ds}{\Ds-1}}.
\end{equation}
Let we call
\[
Q_n(r):=\sup_{x\in \Omega}\int_{B_r(x)}(|\nabla z_n|^2+\lambda|z_n|^2)\di x,
\]
choose $x \in \Omega$ and consider
\[
\tilde{z}_n:=R_n^{\frac{2-N}{2}}z_n\left(\frac{x}{R_n}+x_n\right).
\]
It results
\[
\tilde{Q}_n(1)=\sup_{\substack {x\in \setR\\ x/R_n+x_n \in \Omega}}\int_{B_1(x)}(|\nabla \tilde{z}_n|^2+\lambda|\tilde{z}_n|^2)\di x=\int_{B_1(0)}(|\nabla \tilde{z}_n|^2+\lambda|\tilde{z}_n|^2)\di x=\frac{1}{2L}S_{H,L}^{\frac{\Ds}{\Ds-1}}
\]
where $L\in \mathbb{N}$ is such that $B_2(0)$ is covered by $L$ balls of radius $1$. It is easy to see that, from $\eqref{7F}$, $R_n\ge R_0>0$ uniformly in $n$.
Now denote with $\tilde{\Omega}_n:=\left \{x\in \setR:\frac{x}{R_n}+x_n\in \Omega \right \}$ so that we can regard $\tilde{z}_n\in H_0^1(\tilde{\Omega}_n)\subset H^{1}(\setR)$.
It holds
\[
\|\tilde{z}_n\|_\lambda^2=\|z_n\|_\lambda^2\rightarrow m \Big(\frac{2\cdot \Ds}{\Ds -1} \Big)<\infty,
\]
so we can assume $\tilde{z}_n \rightharpoonup z$ in $H^{1}(\setR)$. 

Proceeding as in \cite{S}*{Lemma $3.3$}, replacing $\beta^*$ with $m_*^\Omega$ and $E_0$ with $I_*^\Omega$, we obtain that $\tilde{z}_n \rightarrow z$ in $H^1(\Omega')$ for any $\Omega' \subset\subset \setR$.
 
Now we distinguish two cases:
\begin{enumerate}
\item[$1)$]$R_n \dist (x_n, \partial \Omega)\le c <\infty$ uniformly.

In this case, less than a rotation, we may suppose that the sequence $\tilde{\Omega}_n$ exhausts 
\[
\tilde{\Omega}_\infty=\mathbb{R}_+^N=\{x=(x_1,\cdots, x_N); x_1>0 \}.
\]
\item[$2)$]$R_n \dist (x_n, \partial \Omega)\rightarrow \infty$ that implies $\tilde{\Omega}_n\rightarrow \tilde{\Omega}_\infty=\setR.$
\end{enumerate}
In both cases, for any $\varphi \in C_0^\infty(\tilde{\Omega}_\infty)$, we get $\varphi \in C_0^\infty(\tilde{\Omega}_n)$ for $n$ large, so
\[
\langle \varphi, (I_*^\Omega)'(z,\tilde{\Omega}_\infty)\rangle=\lim_{n\rightarrow \infty}\langle \varphi, (I_*^\Omega)'(\tilde{z}_n,\tilde{\Omega}_n)\rangle=0
\]
for all these $\varphi$. Hence $z\in H_0^{1}(\tilde{\Omega}_\infty)$ and weakly solves $\eqref{PLRN}$ on $\tilde{\Omega}_\infty$. But in the case $1)$ by \cite{GY}*{Theorem $1.5$} and \cite{EL}*{Theorem I.1.}, $z\equiv 0$. Therefore it has to be $R_n \dist (x_n, \partial \Omega)\rightarrow \infty$.

We conclude the proof letting $\varphi \in C_0^\infty(\setR)$ such that $0\le \varphi \le 1$, $\varphi \equiv 1$ in $B_1(0)$, $\varphi\equiv 0$ outside $B_2(0)$ and
\[
w_n(x)=z_n(x)-R_n^{\frac{N-2}{2}}z(R_n(x-x_n))\cdot \varphi (\bar{R}_n(x-x_n))\in H_0^{1}(\Omega),
\]
where $\{R_n \}_{n\in \N}$ is chosen such that $\tilde{R}_n:=R_n(\bar{R}_n)^{-1}\rightarrow \infty$ as $n \rightarrow +\infty$, i.e.\
\[
\tilde{w}_n=R_n^{\frac{2-N}{2}}w_n\Big(\frac{x}{R_n}+x_n\Big)=\tilde{z}_n(x)-z(x)\varphi\Big(\frac{x}{\tilde{R}_n}\Big).
\]
Then we may proceed as in \cite{ST}*{Proof of Lemma $3.3$} obtaining $\tilde{w}_n=\tilde{z}_n-z+o(1)$ with $o(1)\rightarrow 0$ as $n\rightarrow +\infty$ and using $\eqref{7C}$, $\eqref{7CC}$ and $\eqref{7D}$ we get
\[
I_*^\Omega(w_n)=I_*(\tilde{w}_n)=I_*(\tilde{z}_n)-I_*^\Omega(z)+o(1)
\]
and
\begin{equation}
\begin{aligned}
\|(I_*^\Omega)'(w_n,\Omega)\|_{H^{-1}}&=\|I_*'(\tilde{w}_n,\tilde{\Omega}_n)\|_{H^{-1}}\le \|I_*'(\tilde{z}_n,\tilde{\Omega}_n)\|_{H^{-1}}\\
&+\|(I_*^\Omega)'(z,\setR)\|_{H^{-1}}+o(1)=\|(I_*^\Omega)'(z_n,\Omega)\|_{H^{-1}}+o(1)\rightarrow 0
\end{aligned}
\end{equation}
as $n\rightarrow \infty$ and this conclude the proof.
\end{proof}

Applying Lemma \ref{lemmaausiliario} to sequence $z_n^1:=v_n-v$, $z_n^j:=~v_n-v-\sum_{i=1}^{j-1}{v_n^i}=z_n^{j-1}-v_n^{j-1}$ with $j>1$ and
\[
v_n^i(x):=(R_n^i)^{\frac{N-2}{2}}v^i(R_n^i(x-x_n^i)),
\]
by induction we get
\begin{equation}
\begin{aligned}
I_*^\Omega(z_n^i)&=I_*^\Omega(v_n)-I_*^\Omega(v)-\sum_{i=1}^{j-1}{I_*(v^i)}+o(1)\\
&\le I_*^\Omega(v_n)-(j-1)m^*+o(1).
\end{aligned}
\end{equation}
Note that for large $j$, the latter will be negative, so by Lemma \ref{lemmaausiliario} the induction will stop after some index $k>0$. For this index we have
\[
z_n^{k+1}=v_n-v-\sum_{j=1}^k{v_n^j}\rightarrow 0
\]
 and
\[
I_*^\Omega(v_n)-I_*^\Omega(v)-\sum_{j=1}^k{I_*(v^j)}\rightarrow 0,
\]
so we conclude the proof.

\end{document}